\documentclass{amsart}
\usepackage{graphicx}
\usepackage{marginnote}
\usepackage{hyperref,pinlabel,xcolor,enumitem,comment}
\usepackage[margin=1.25in]{geometry}

\hypersetup{
    colorlinks=true,
    linkcolor=magenta,
    citecolor=magenta,
    filecolor=magenta,      
    urlcolor=cyan,
    }
    \definecolor{purple}{rgb}{.5, 0, 0.5}

\newtheorem{theorem}{Theorem}[section]
\newtheorem{corollary}[theorem]{Corollary}
\newtheorem{lemma}[theorem]{Lemma}
\newtheorem{proposition}[theorem]{Proposition}
 
\newtheorem{claim}[theorem]{Claim} 
\newtheorem*{mainthm}{Theorem \ref{thm:main}}
\newtheorem*{cor1}{Corollary \ref{cor1}}
\newtheorem*{cor2}{Corollary \ref{cor2}}
\theoremstyle{definition}
\newtheorem{definition}[theorem]{Definition}
\newtheorem{problem}[theorem]{Problem}
\newtheorem{remark}[theorem]{Remark}

\theoremstyle{definition}
\newtheorem{question}[theorem]{Question}
\newtheorem*{question*}{Question} 
\newtheorem{questions}[theorem]{Questions}

\newtheorem{conventions}[theorem]{Conventions}

\newtheorem{observation}[theorem]{Observation}
\newcommand{\D}{\mathcal{D}}

\makeatother

\newcommand{\blue}[1]{\textcolor{blue}{#1}}

\DeclareMathOperator{\MS}{MS}
\DeclareMathOperator{\IS}{IS}

    \author[Seungwon Kim]{Seungwon Kim}
    \address{Sungkyunkwan University\\Suwon, Gyeonggi, 16419 Republic of Korea}
    \email{seungwon.kim@skku.edu}
    
    \author[Maggie Miller]{Maggie Miller}
    \address{University of Texas at Austin\\Austin, TX, 78712 USA}
    \email{maggie.miller.math@gmail.com}

  \author[Jaehoon Yoo]{Jaehoon Yoo}
    \address{Indiana University Bloomington\\Bloomington, IN, 47405 USA}
    \email{jaehyoo@iu.edu}

\thanks{SK and JY are supported by a National Research Foundation of Korea (NRF) grants funded by the Korean government (MSIT) (No.2022R1C1C2004559). MM is partially supported by a Clay Research Fellowship and NSF grant DMS-2404810.}
    
\title[Alternating links and surfaces in $B^4$]{Non-split, alternating links bound unique Seifert surfaces in the 4-ball}

\begin{document}

\begin{abstract}
We show that any two same-genus, oriented, boundary parallel surfaces bounded by a non-split, alternating link into the 4-ball are smoothly isotopic fixing boundary. In other words, any same-genus Seifert surfaces for a non-split, alternating link become smoothly isotopic fixing boundary once their interiors are pushed into the 4-ball. We conclude that a smooth surface in $S^4$ obtained by gluing two Seifert surfaces for a non-split alternating link is always smoothly unknotted. 
\end{abstract}


\maketitle


\section{Introduction}\label{sec:intro}
   In this paper, we discuss isotopy of boundary parallel surfaces in the 4-ball.
   
   \begin{conventions}
       All surfaces, knots, and links are oriented; all embeddings and isotopies are smooth.  A {\emph{Seifert surface}} will always denote an oriented, not necessarily connected surface embedded in $S^3$ that has no closed components. The boundary of such a surface $S$ is a link $L$, and we say that $S$ is a Seifert surface for $L$.
       
   Typically, we will obtain surfaces first as Seifert surfaces in $S^3$. When we say that two surfaces in $S^3$ are isotopic rel.\ boundary in $B^4$, we implicitly mean that the surfaces become isotopic rel.\ boundary after their interiors are pushed slightly into the interior of $B^4$. We prove the following main theorem. 
   \end{conventions}

   \begin{theorem}\label{thm:main}
Any two same-genus Seifert surfaces for a non-split, alternating link are smoothly isotopic rel.\ boundary in $B^4$.
\end{theorem} 

\begin{corollary}\label{cor1}
    Up to isotopy rel.\ boundary in $B^4$, a non-split, alternating link $L$ bounds a unique minimal genus Seifert surface $S$. Any other Seifert surface for $L$ is isotopic rel.\ boundary in $B^4$ to a connected sum of $S$ and some number of standard tori in $S^4$.
\end{corollary}

\begin{corollary}\label{cor2}
    Let $S_1, S_2$ be Seifert surfaces for a non-split, alternating link $L$ in $S^3$. Viewing $S^3$ as an equator of $S^4$, push the interiors of $S_1$ and $S_2$ off $S^3$ toward opposite normal directions of $S^3$ to obtain surfaces $S_1^+, S_2^-$ in $S^4$ with the same boundary but disjoint interior. The closed surface $S_1^+\cup S_2^-$ bounds a handlebody smoothly embedded in $S^4$, i.e.,\ is smoothly unknotted.
\end{corollary}

    There are certain cases where a knot or link is already known to have a unique minimal genus Seifert surface up to isotopy rel.\ boundary in $B^4$. For example, a fibered link $L$ bounds a unique minimal genus Seifert surface up to isotopy rel.\ boundary in $S^3$. This condition is also known to hold for some other simple families of knots, see e.g.,\ \cite{lyon-simple-unique,kobayshi-uniqueness}. Theorem \ref{thm:main} is the first theorem proving uniqueness of minimal genus Seifert surfaces for a family of knots up to isotopy rel.\ boundary in $B^4$ when uniqueness does not hold in $S^3$. 
    
    Many surfaces that are distinct in $S^3$ become isotopic rel.\ boundary in $B^4$. See \cite[Section 2.1]{hayden-kim-miller-park-sundberg} for a detailed discussion of constructions in the literature.  However, the first two authors along with Hayden, Park, and Sundberg constructed many examples of connected, minimal genus Seifert surfaces with common boundary that do not become isotopic rel.\ boundary in $B^4$ \cite{hayden-kim-miller-park-sundberg}. 
   Even more examples of this behavior have since been produced by Aka--Feller--A.\ B.\ Miller--Wieser \cite{aka-feller-miller-wieser}. Thus in general, we do not expect a knot or link to bound a unique minimal genus Seifert surface, even up to isotopy rel.\ boundary in $B^4$.

   In Theorem \ref{thm:main}, we focus on the setting of alternating links. A non-split, alternating link may bound multiple minimal genus Seifert surfaces up to isotopy rel.\ boundary in $S^3$. As an explicit example, Trotter showed the knot {\tt{9$_4$}} bounds at least two distinct minimal genus Seifert surfaces \cite[Theorem 2.6]{trotter}. Roberts further constructed a family of alternating knots $\{K_n\}_{n\in\mathbb{N}}$ such that $K_n$ bounds at least $2^{2n-1}$ distinct minimal genus Seifert surfaces \cite{roberts}. However, we conclude from Theorem \ref{thm:main} that these Seifert surfaces are isotopic rel.\ boundary in $B^4$.

Corollary \ref{cor2} is particularly interesting given the context of the following open problem.

\begin{question}\label{q:seifertunknotting}
   Let $S_1, S_2$ be Seifert surfaces for a knot $K$ in $S^3$. Viewing $S^3$ as an equator of $S^4$, push the interiors of $S_1$ and $S_2$ off $S^3$ toward opposite normal directions of $S^3$ to obtain surfaces $S_1^+, S_2^-$ in $S^4$ with the same boundary but disjoint interior. Is the closed surface $S_1^+\cup S_2^-$ necessarily  unknotted? 
\end{question}

The answer to Question \ref{q:seifertunknotting} is ``yes," in the topological category by Conway--Powell \cite{conwaypowell}. (As a consequence, it was already known by \cite{conwaypowell} that the surface $S_1^+\cup S_2^-$ in Corollary \ref{cor2} bounds a locally flat handlebody in $S^4$.) In the smooth category, Corollary \ref{cor2} answers Question \ref{q:seifertunknotting} affirmatively in the case that $K$ is alternating.

Question \ref{q:seifertunknotting} is a well-known subquestion of the smooth unknotting conjecture, which posits that any smooth, oriented surface in $S^4$ whose complement has cyclic fundamental group is smoothly unknotted. Note that an affirmative answer to Question \ref{q:seifertunknotting} for would {\emph{not}} imply that same-genus Seifert surfaces $\Sigma_1, \Sigma_2$ for a non-alternating knot $K$ are isotopic rel.\ boundary once pushed into $B^4$, which is generally false. Even if $\Sigma_1$ and $\Sigma_2$ are not isotopic rel.\ boundary in $B^4$, their union as an embedded surface in $S^4$ may be smoothly unknotted, e.g.,\ the examples of \cite{hayden-kim-miller-park-sundberg} and \cite{aka-feller-miller-wieser} have this property. 
See discussion in \cite[Example 2.1]{hayden-kim-miller-park-sundberg}.  

\section{The Kakimizu complex}\label{sec:background}

The proof of Theorem \ref{thm:main} hinges on understanding the space of incompressible Seifert surfaces for a non-split, alternating link. This space (for a general link) is called the {\emph{Kakimizu complex}}. In the literature, it is common to restrict to minimal genus surfaces -- for the sake of discussion, we give both definitions. 

\begin{definition}[Kakimizu {\cite{Kakimizu_paper}}]
    Let $L$ be a non-split link in $S^3$. The {\emph{Kakimizu complexes of $L$}}, denoted $\MS(L)$ and $\IS(L)$, are simplicial complexes constructed as follows.
    \begin{itemize}
   \item  The vertices of $\MS(L)$ are in bijection with proper isotopy (rel.\ boundary) classes of minimal genus Seifert surfaces for $L$.
   \item The vertices of $\IS(L)$ are in bijection with proper isotopy (rel.\ boundary) classes of incompressible Seifert surfaces for $L$.
    \item In each of $\MS(L), \IS(L)$, for $n>0$, there is an $n$-simplex with vertices $v_0,\ldots, v_n$ exactly when the isotopy classes corresponding to $v_0,\ldots, v_1$ admit Seifert surface representatives with mutually disjoint interiors.
    \end{itemize}
\end{definition}

The notations ``$\MS(L)$," and $``\IS(L)$," chosen by Kakimizu, are meant to remind the reader that these simplicial complexes are associated to {\bf{m}}inimal {\bf{s}}urfaces or {\bf{i}}ncompressible {\bf{s}}urfaces bounded by the link $L$.

\begin{theorem}[Kakimizu {\cite{Kakimizu_paper}}]\label{thm:kakimizu}
For any non-split link $L$, the complexes $\MS(L)$ and $\IS(L)$ are each connected.
\end{theorem}

The methods of Kakimizu suggest the following strengthened version of Theorem \ref{thm:kakimizu}.

\begin{definition}
    Fix an integer $k$ and a non-split link $L$ in $S^3$. Let $\IS_k(L)$ be the simplicial subcomplex of $\IS(L)$ that contains an $n$-simplex with vertices $v_0,\ldots, v_n$ exactly when incompressible Seifert surfaces representing $v_0,\ldots, v_n$ each have Euler characteristic at least $k$.
\end{definition}

\begin{theorem}\label{thm:filtration}
Given a non-split link $L$, for each integer $k$ the subcomplex $\IS_k(L)$ is connected.
\end{theorem}

Note that a Seifert surface for $L$ restricts to a properly embedded, oriented surface in $S^3\setminus\nu(L)$ that has no closed components nad has boundary a longitude on each component of $\partial S^3\setminus\nu(L)$. For convenience, we will refer interchangeably to such properly embedded surfaces and Seifert surfaces for $L$.
 
We will make use of the following definition, used by Kakimizu in \cite{Kakimizu_paper}. 
\begin{definition}\label{def:coveringdistance}
Let $S,S'$ be incompressible Seifert surfaces for a non-split link $L$. Let $S_0$ and $S'_0$ respectively be lifts of $S'$ to the infinite cyclic cover $M$ of $S^3\setminus\nu(L)$. For each integer $i$, let $S_i$ denote the translate of $S_0$ by the action of $i\in\mathbb{Z}$ and let $E_i$ denote the closure of the region of $M$ bounded by $-S_i\cup S_{i+1}$. Similarly let $S'_j$ denote the translate of $S'_0$ by the action of $j\in\mathbb{Z}$ and let $E'_j$ denote the closure of the region of $M$ bounded by $-S'_j\cup S'_{j+1}$. Isotope $S'$ (and hence equivariantly isotope $\{S'_j\}$) so as to minimize the number of integers $i$ for which $E'_0\cap E_i\neq\emptyset$. Let $a,b$ be integers with $E'_0\cap E_a,E'_0\cap E_b\neq\emptyset$ and $E'_0\cap E_i=\emptyset$ for $i>b$ and for $i<a$. The {\emph{covering distance}} $d(S,S')$ is defined to be $b-a$.
\end{definition}

Kakimizu \cite[Proposition 3.1(1)]{Kakimizu_paper} proved that the covering distance of $S$ and $S'$ is equal to the minimum length of a path between vertices corresponding to $S,S'$ in $\IS(L)$. Even without this fact (which we will not directly employ in this paper, so as to keep our arguments relatively self-contained), we might make the following observation.

\begin{observation}[Kakimizu \cite{Kakimizu_paper}]\label{observation:coveringzero}
Incompressible Seifert surfaces $S$ and $S'$ for $L$ satisfy $d(S,S')=0$ if and only if $S, S'$ are properly isotopic rel.\ boundary: in Definition \ref{def:coveringdistance}, if $d(S,S')=0$ then up to isotopy we have $E'_0\subset E_i$ for some $i$. Since the boundaries of $E'_0$ and $E_i$ are lifts of $S', S$ respectively, we must have $E'_0=E_i$ and hence $S,S'$ are isotopic rel.\ boundary in $S^3$.
\end{observation}

We now prove Theorem \ref{thm:filtration} via the following lemma, which mirrors that of \cite[Theorem 2.1]{Kakimizu_paper} and closely follows the same proof, taking care to consider the Euler characteristic of the produced surfaces.

\begin{lemma}\label{lem:kakimizuinduct}
Let $S,S'$ be non-isotopic, incompressible Seifert surfaces for a non-split link $L$. Let $n:=d(S,S')$. Then there is an incompressible Seifert surface $S''$ for $L$ such that $S''$ and $S'$ do not intersect in their interiors, $d(S,S'')<n$, and $\chi(S'')\ge\min\{\chi(S),\chi(S')\}$.
\end{lemma}

The following proof of Lemma \ref{lem:kakimizuinduct} is illustrated in Figure \ref{fig:lemma_filtration}.

\begin{figure}
\labellist
\pinlabel{$S_0$} at 32 -10
\pinlabel{\textcolor{red}{$S_1$}} at 110 -10
\pinlabel{$S_2$} at 188 -10
\pinlabel{$S_3$} at 265 -10
\pinlabel{$S_4$} at 344 -10
\pinlabel{$S'_{-1}$} at 68 160
\pinlabel{\textcolor{blue}{$S'_0$}} at 142 160
\pinlabel{$S'_1$} at 218 160
\pinlabel{$S'_2$} at 299 160
\pinlabel{$R$} at 128 40
\pinlabel{$R'$} at 52 40
\endlabellist
\vspace{.1in}
    \includegraphics[width=100mm]{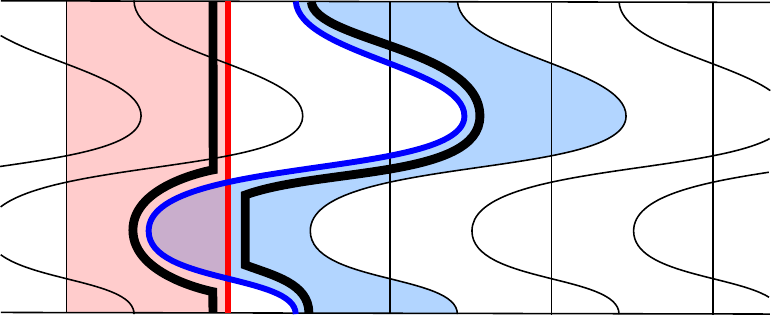}
    \vspace{.1in}
    \caption{A schematic of the infinite cyclic cover $M$ of $E:=S^3\setminus\nu(L)$. We indicate lifts $S_i$ of $S$ and $S'_j$ of $S'$. The region between $S_i,S_{i+1}$ is $E_i$ and the region between $S'_j,S'_{j+1}$ is $E'_j$. In this schematic, $d(S,S')=3$ as $E'_0$ intersects $E_i$ for the values $i\in\{0,1,2,3\}$. In the proof of Lemma \ref{lem:kakimizuinduct}, we obtain surfaces $R\subset E'_0$ and $R'\subset E_0$ whose union is isotopic to the cut-and-paste sum of $S'_0$ and $S_1$. It follows that one of $R,R'$ has Euler characteristic bounded below by $\min\{\chi(S),\chi(S')\}$. Fully compressing that surface (in $E'_0$ or $E_0$ accordingly) yields an incompressible surface that projects to the desired Seifert surface $S''$.}\label{fig:lemma_filtration}
\end{figure}
\begin{proof}[Proof of Lemma \ref{lem:kakimizuinduct}]

    Let $S_0$ be a lift of $S$ to $M$ and let $S_i$ denote the translate of $S_i$ under the action of $i\in\mathbb{Z}$. Let $E_i$ denote the closure of the region of $M$ bounded by $-S_i\cup S_{i+1}$. Since $n=d(S,S')$, by definition of covering distance we may isotope $S'$ and choose a lift $S'_0$ to $M$ such that for some integer $a$, $S'_0$ intersects $E_a$ and $E_{a+n-1}$ but does not intersect $E_i$ for any $i<a$ or $i> a+n-1$. Relabel the lifts $S_i$ if necessary so that $a=0$. Perturb $S'$ if necessary so that intersections of $S$ and $S'$ in $E$ are transverse and the interiors of $S,S'$ intersect in closed curves. Since $S,S'$ are each incompressible, each component of $S\cap S'$ is either essential in both $S$ and $S'$ or inessential in both $S,S'$. Remove any inessential intersections using a standard innermost circle surgery argument.

   Similarly, let $S'_j$ be the translate of $S'_0$ under the action of $j\in\mathbb{Z}$ and let $E'_j$ denote the closure of the region of $M$ bounded by $-S'_j\cup S'_{j+1}$.  Then $E'_0\cap E_0\neq\emptyset$ and $E'_i\cap E_0=\emptyset$ whenever $i>0$. Let $X$ be a regular neighborhood of $S'_0\cup(E_0\cap E_0')$ in $E_0'$. Since $S'_0\cap E_n=\emptyset$, we also have $X\cap E_n=\emptyset$. Let $R$ be the closure of $\partial X\setminus\partial E'_0$.

   We simultaneously produce another surface, exchanging the roles of $S,S'$. That is, we note that $E_j\cap E'_0=\emptyset$ whenever  $j<0$. Let $X'$ be a regular neighborhood of $S_1\cup(E_0\cap E'_0)$ in $E_0$. Since $S_1\cap E'_{-n}=\emptyset$, we also have $X'\cap E'_{-n}=\emptyset$. Let $R'$ be the closure of $\partial X'\setminus\partial E_0$.

   The union of $R$ and $R'$ is isotopic to the cut-and-paste sum of  $S_1$ and $S'_0$. (That is, the union of $R$ and $R'$ is isotopic to the embedded surface obtained by resolving $S_1\cup S'_0$ at intersections while respecting orientation. Note that if $S_1, S'_0$ are disjoint, as may be the case if $n=1$, then $R$ is a parallel copy of $S'_0$ while $R'$ is a parallel copy of $S_1$.) 
   We conclude $\chi(R)+\chi(R')=\chi(S)+\chi(S')$, so at least one of $\chi(R),\chi(R')$ is bounded below by $\min\{\chi(S),\chi(S')\}$. By potentially exchanging the roles of $S$ and $S'$ (and hence also $R$ and $R'$), assume without loss of generality that $\chi(R)\ge\min\{\chi(S),\chi(S')\}$. Since $S,S'$ intersect in essential curves, any closed components of $R$ are positive genus. Delete any closed components; we still have $\chi(R)\ge\min\{\chi(S),\chi(S')\}$.

Since $R\subset X\subset E_0'$, $R$ does not intersect any $S_i'$. Then $p(R)\cap S'=\emptyset$,  where $p(R)$ is the projection of $R$ to $E$ (and hence a Seifert surface for $L$). Moreover, $R$ does not intersect $E_j$ with $j<0$ or $j>n$. We already remarked that $X$ does not intersect $E_n$, so $R$ also does not intersect $E_n$. We conclude that if $R$ is incompressible, then the Seifert surface $p(R)$   satisfies $d(S,p(R))<n$, satisfying the claim (with $S'':=p(R)$).

Unfortunately, the surface $R$ may fail to be incompressible. Since $S,S'$ intersect in essential circles and $R$ is a subset of the cut-and-paste sum of $S_1$ and  $S'_0$, the surface $R$ does not include any 2-sphere components. 
Since $\partial E'_0$ is incompressible in $M$, if a surface in the interior of $E'_0$ is incompressible in $E'_0$ then it is also incompressible in $M$. 
Suppose $R$ is compressible in $M$ and hence in $E'_0$. Let $D$ be a compressing disk for $R$ in $E'_0$. 
Since $S_n$ is incompressible, again by a standard innermost circle argument we may assume $D$ does not intersect $S_n$. Then compressing $R$ along $D$ yields a new surface in $E'_0$ that has no 2-sphere components and does not intersect $E_n$. Repeat until obtaining a surface that is incompressible in $E'_0$ and hence also incompressible in $M$. Delete any closed components and call the result $\tilde{R}$. Since $\tilde{R}$ is obtained from $R$ by compression and deletion of positive-genus components, $\chi(\tilde{R})\ge\chi(R)\ge\min\{\chi(S),\chi(S')\}$. Let $S''$ denote the incompressible Seifert surface $p(\tilde{R})$. Since $\tilde{R}\subset E'_0\setminus E_n$, we find that the interiors of $S''$ and $S'$ are disjoint and $d(S,S'')<n$ as desired.
\end{proof}

Theorem \ref{thm:filtration} follows easily from Lemma \ref{lem:kakimizuinduct}.
\begin{proof}[Proof of Theorem \ref{thm:filtration}]
Let $v,v'$ be vertices in $\IS_k(L)$ for some $k$. We will construct a path in $\IS_k(L)$ from $v'$ to $v$, implying that $\IS_k(L)$ is connected. 
Each of $v,v'$ is represented by an incompressible Seifert surface $S,S'$ (respectively) with $\chi(S),\chi(S')\ge k$. Let $n:=d(S,S')$. By Lemma \ref{lem:kakimizuinduct}, there exists an incompressible Seifert surface $F_1$ (taking the place of $S''$) for $L$ representing the vertex $v_1\in\IS(L)$ such that 
$\chi(F_1)\ge\min\{\chi(S),\chi(S')\}\ge k$ (so in particular $v_1\in \IS_k(L)$), 
    the vertices $v_1,v'$ are adjacent in $\IS(L)$, and
   $d(F_1,S)<n$.

If $d(F_1,S)=0$, then by Observation \ref{observation:coveringzero} the surfaces $F_1, S$ are isotopic rel.\ boundary and we conclude there is a path $(v',v_1,v)$ in $\IS_k(L)$. Otherwise, repeat the above argument replacing $S'$ with $F_i$ to obtain a surface $F_{i+1}$ representing a vertex $v_{i+1}\in\IS(L)$ such that 
    $\chi(F_{i+1})\ge\min\{\chi(S),\chi(F_i)\}\ge \min\{\chi(S),\chi(S')\}\ge k$, (so in particular $v_{i+1}\in \IS_k(L)$), 
    the vertices $v_{i+1},v_i$ are adjacent in $\IS(L)$, and 
   $d(F_{i+1},S)<d(F_i,S)$. 
Repeat until obtaining a surface $F_a$ with $d(F_a,S)=0$ and hence (again by Observation \ref{observation:coveringzero}) the surface $F_a$ is isotopic rel.\ boundary to $S$. We conclude there is a path $(v',v_1,v_2,\ldots, v_a,v)$ in $\IS_k(L)$.
\end{proof}

The following corollary follows from Kakimizu \cite[Proposition 3.1(1)]{Kakimizu_paper}, but we prove it here for completeness. 
\begin{corollary}\label{cor:distance1}
    Let $S,S'$ be incompressible Seifert surfaces for a link $L$. Then $d(S,S')=1$ if and only if $S$ and $S'$ represent adjacent vertices in $\IS(L)$. 
    \end{corollary}

    \begin{proof}
        If $S,S'$ represent adjacent vertices, then up to isotopy rel.\ boundary the interiors of $S,S'$ are disjoint and hence $d(S,S')=1$.

        On the other hand, if $d(S,S')=1$, then it follows from Lemma \ref{lem:kakimizuinduct} that there is a third Seifert surface $S''$ such that $S''$ is disjoint from $S'$ in its interior and $d(S,S'')<1$. Then $d(S,S'')=0$ and Observation \ref{observation:coveringzero} shows that $S,S''$ are isotopic rel.\ boundary in $S^3$. We conclude that up to isotopy rel.\ boundary in $S^3$, the surfaces $S,S'$ have disjoint interior and thus represent adjacent vertices in $\IS(L)$.
    \end{proof}

\begin{remark}\label{genusremark}
     In \cite{prz-schultens}, Przytycki--Schultens give a similar but inequivalent definition of $\MS(L)$ -- in their notation, $\MS(E(L))$ -- that differs from $\MS(L)$ when $L$ admits minimal genus, disconnected surfaces. They show that $\MS(E(L))$ is contractible. If $L$ is non-split and alternating, then every oriented surface bounded by $L$ is connected: consider the classic exercise \cite[Proposition 4.8]{lickorish} showing that an incompressible positive-genus surface $F$ in the complement of $L$ admits a compressing disk intersecting $L$ in a single point. This would be violated by $F:=\partial\nu(S_1)$ if $S_1\sqcup S_2$ is a disconnected surface bounded by $L$ (possibly after compressing $F$ until reaching an incompressible splitting surface for $L$), since every simple closed curve of $F$ links $L$ zero times. (Alternatively, one could cite \cite[Corollary 5.2]{A_classification_of_spanning_surfaces_for_alternating_links}.)
    Therefore, in the settings considered in this paper, we can ignore this distinction (and hence $\MS(L)$ is contractible, although we do not require more than connectivity). We similarly need not distinguish between ``minimal genus" and ``maximal Euler characteristic," since all surfaces are connected. It remains open whether $\IS(L)$ is contractible (see \cite[\S1]{prz-schultens} for discussion).
    \end{remark}
    
In the case that $L$ is an alternating link, there is even more known about surfaces bounded by or containing $L$. It is a well-known theorem of \cite{crowell_genus_alt_link, murasugi_genus_of_alt_knot} that Seifert's algorithm applied to an alternating diagram yields a minimal genus Seifert surface. However,  a minimal genus Seifert surface for an alternating link need {\emph{not}} arise from Seifert's algorithm applied to an alternating diagram \cite[Theorem 1.2]{hirasawa_sakuma_minimal_genus}{\footnote{Thanks to Thomas Kindred for pointing out the following subtlety to the authors: while \cite{hirasawa_sakuma_minimal_genus} produces Seifert surfaces not arising from Seifert's algorithm applied to any {\emph{alternating}} diagram, it is unknown whether an alternating link can admit a minimal genus Seifert surface not arising from Seifert's algorithm applied to {\emph{some}} (not necessarily alternating) diagram.}}, and in general it is not clear what, if any, relationship is held between distinct minimal genus Seifert surfaces for an alternating link.

\begin{definition}
    Given a closed surface $F$ in $S^3$ containing a link $L$, we say that $F$ is {\emph{essential with respect to $L$}} when  $F\cap (S^3\setminus\nu(L))$ is essential in $S^3\setminus\nu(L)$. When $L$ is clear, we may simply write that $F$ is essential.
\end{definition}

We will make use of the following key lemma proved by Kindred \cite{kindred_representativity}.

    \begin{lemma}[Kindred \cite{kindred_representativity} (Crossing Tube Lemma)]\label{C.T.L}
        Given a nontrivial, reduced alternating diagram $\D$ of a non-split link $L$ and a closed, essential surface $F$ containing $L$, 
        there exists an isotopy of $S^3$ fixing $L$ after which $F$ has a standard tube above some crossing $c$ of $\D$ as in Figure \ref{fig:standardtube}. That is, $F$ admits a compressing disk lying above a small neighborhood of $c$ and that intersects $L$ in the two points of $L$ above $c$. 
    \end{lemma}

    \begin{figure}
    \centering
    \includegraphics[width=0.7\linewidth]{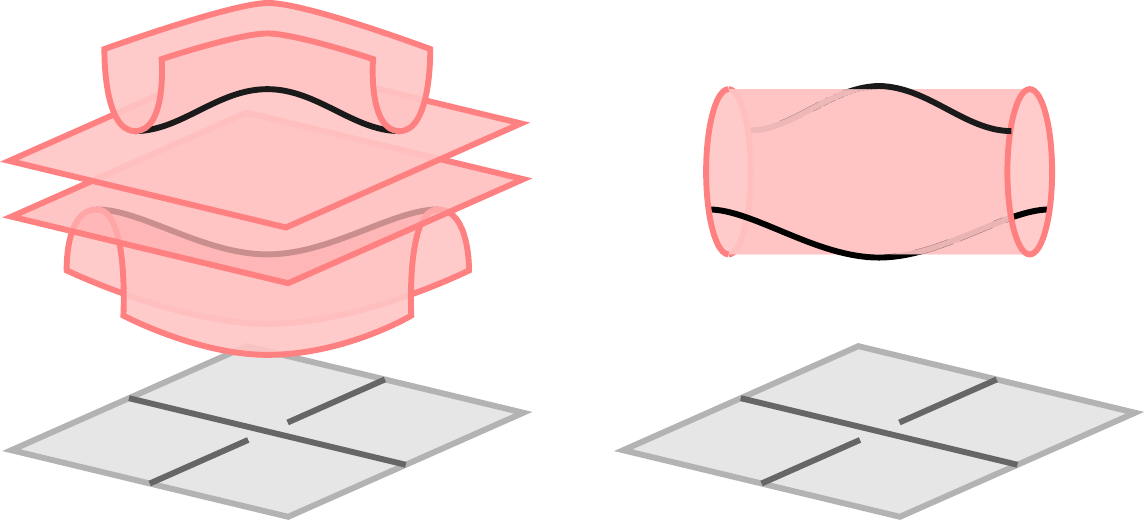}
    \caption{Left: a generic essential surface $F$ containing a link $L$. We feature the portion of $F$ lying above a neighborhood of some crossing in the diagram $\D$ (shown at the bottom of the figure). This subset of $F$ includes two disks that contain two arcs of $L$ and potentially other disks that do not meet $L$. Right: above this crossing, the surface $F$ includes a standard tube containing the two strands that project to the crossing.}
    \label{fig:standardtube}
\end{figure}

\section{Proof of Theorem \ref{thm:main}}

 We will reduce to the setting of incompressible surfaces using the following lemma, which is based on a commonly used fact in the study of surfaces in 4-manifolds. See \cite[Corollary 4]{boyle} for a general statement and proof. Here, when we say a Seifert surface $S_i$ is obtained by compressing another Seifert surface $\Sigma_i$, we mean that we compress $\Sigma_i$ and delete any resulting closed components.
 
\begin{lemma}\label{lem:stab}
    Let $\Sigma$ be a connected Seifert surface for a link $L$. Suppose that $S$ is a connected Seifert surface obtained by compressing $\Sigma$ along a disk (possibly deleting a resulting closed component). Push the interiors of both surfaces into $B^4$ and let $T$ denote the standard torus $S^4$. Then $\Sigma$ is isotopic rel.\ boundary to $S\#(g(\Sigma)-g(S))T$.
\end{lemma}
\begin{proof}
If the compression takes place along a non-separating curve, then $\Sigma$ is obtained from $S$ by attaching one tube, i.e.,\ surgery along a framed arc $\gamma$. Because $S$ is boundary parallel in $B^4$, we have $\pi_1(B^4\setminus S)\cong\mathbb{Z}$ with generator represented by a meridian of $S$. Therefore, any two arcs with endpoints on $S$ and interior in $B^4\setminus S$ are isotopic keeping ends on $S$ and interior disjoint from $S$. In ambient dimension four, the normal bundle $N(\gamma)$ of the arc $\gamma$ rel.\ boundary admits framings in bijection with $\pi_1(SO(3))\cong\mathbb{Z}/2\mathbb{Z}$. Such a framing corresponds with a choice of how to attach a tube to $S$ along $\gamma$: we choose a trivialization $\langle v_1, v_2, v_3\rangle$ of $N(\gamma)$ with $\langle v_1, v_2\rangle$ tangent to $S$ at the ends of $\gamma$ and then surger $S$ along the 3-dimensional 1-handle $\gamma\times\langle v_1,v_2\rangle$. Exactly one of the two potential framings of $N(\gamma)$ yields a 1-handle over which the orientation of $S$ extends, while the other framing yields a 1-handle over which the orientation does not extend. Since $\Sigma$ is oriented, the framing of $\gamma$ is uniquely determined.     We conclude that there is a unique way of attaching an oriented tube to $S$ in $B^4$, one of which would yield $S\# T$, and hence that $\Sigma$ is isotopic rel.\ boundary to $S\# T$.

On the other hand, support that $S$ is obtained by compressing $\Sigma$ along a separating curve and then deleting a closed component $S'$ disjoint from $S$.  Recall that a closed surface in $S^3$ always bounds a handlebody in $B^4$ (see e.g., discussion in the beginning of\ \cite[\S 1.2]{hatcher_3m}: every positive genus surface in $S^3$ is compressible and every smooth $S^2$ in $S^3$ bounds a 3-ball). Let $V$ be a handlebody bounded by $S'$ into $B^4$. Let $\gamma$ be a framed arc connecting $S, S'$ such that surgery along $\gamma$ yields $\Sigma$. Push $S', V, \gamma$ into the interior of $B^4$. Since $\gamma$ is 1-dimensional and disjoint in its interior from $V$ and $\pi_1(B^4\setminus V)=1$, as the nonseparating case we conclude that $\gamma$ is isotopic to any other arc connecting $S$ to $S'=gT$ that is disjoint from $V$ and also that the framing on $\gamma$ is determined by the orientation of $\Sigma$. Therefore, $\Sigma$ is isotopic rel.\ boundary in $B^4$ to  $S\# (g(S')T)$.
\end{proof}

\begin{figure}
\labellist
\pinlabel{$\D$} at 42 68
\pinlabel{\textcolor{red}{$S_1$}} at 65 225
\pinlabel{\textcolor{blue}{$S_2$}} at 188 225
\pinlabel{\textcolor{red}{band}} at 65 105
\pinlabel{\textcolor{blue}{band}} at 192 105
\pinlabel{\textcolor{purple}{$b$}} at 430 110
\pinlabel{$\D'$} at 652 68
\pinlabel{\textcolor{red}{$S'_1:=S_1\setminus b$}} at 655 225
\pinlabel{\textcolor{blue}{$S'_2:=\tilde{S}_2\setminus b$}} at 820 227
\pinlabel{$L'$} at 740 225
\pinlabel{isotopy} at 280 125
\pinlabel{rel.\ $L$} at 280 105
\pinlabel{surger $L$} at 585 125
\pinlabel{along $b$} at 585 105
\endlabellist
    \includegraphics[width=140mm]{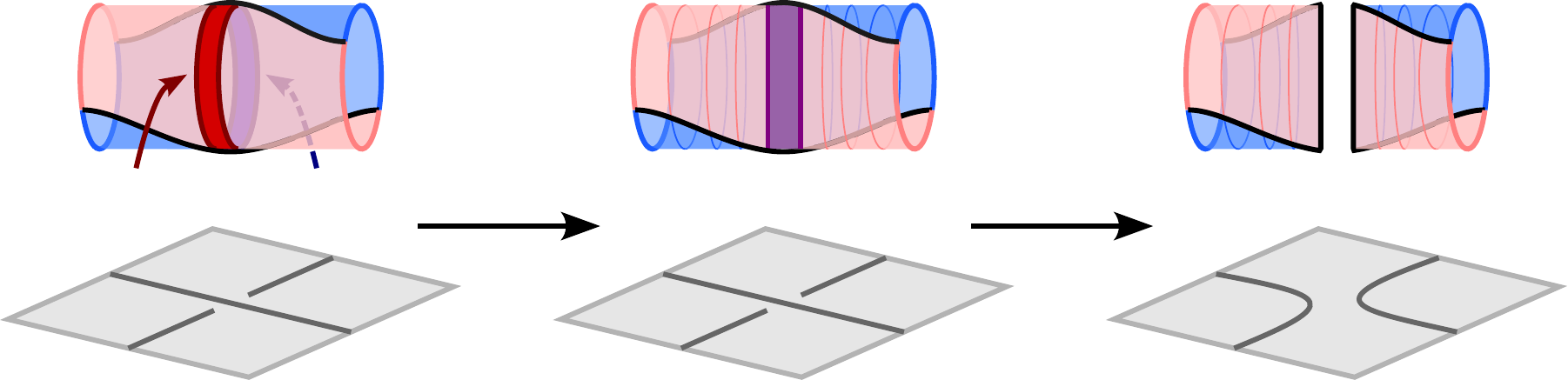}
    \caption{Left: incompressible surfaces $S_1,S_2$ bounded by $L$. By Lemma \ref{C.T.L}, the essential surface $S_1\cup S_2$ includes a standard tube over some crossing $c$ in a reduced, alternating diagram $\D$ for $L$. Center: we isotope $S_1,S_2$ near the standard tube so that they agree in a band $b$. Elsewhere, we add $g(S_1)-g(S_2)$ tubes to $S_2$ to obtain a surface $\tilde{S}_2$ of the same genus as $S_1$. Right: deleting the band $b$ from $S_1,\tilde{S}_2$ yields surfaces $S'_1,S'_2$ bounded by a link $L'$.}\label{fig:proof}
\end{figure}

We will make use of the following subcase of Theorem \ref{thm:main} in the proof of the main theorem.






\begin{lemma}\label{lem:distance1}
    Let $\Sigma_1,\Sigma_2$ be same-genus Seifert surfaces for a non-split, alternating link. Suppose that $\Sigma_1,\Sigma_2$ can be fully compressed to incompressible Seifert surfaces $S_1,S_2$ (respectively) with $d(S_1,S_2)\le 1$. Then $\Sigma_1$ and $\Sigma_2$ are isotopic rel.\ boundary in $B^4$.
\end{lemma}

In the proof of Lemma \ref{lem:distance1}, we use the following easy fact.

\begin{proposition}\label{prop:stilldisjoint}
    If $\Sigma_1, \Sigma_2$ are Seifert surfaces for a non-split alternating link $L$ and $\Sigma_1,\Sigma_2$ have disjoint interior, then $\Sigma_1,\Sigma_2$ can be compressed to incompressible Seifert surfaces $S_1, S_2$ (respectively) that also have disjoint interior.
\end{proposition}
\begin{proof}
    Suppose $D$ is a compressing disk for $\Sigma_1$. Let $C$ be a curve of intersection in $D\cap\Sigma_2$ that is innermost on $D$, so $C$ bounds a disk $D'\subset D$ whose interior is disjoint from $\Sigma_2$. Compress $\Sigma_2$ along $D'$ and delete any closed components formed. Repeat until the interior of $D$ is disjoint from $\Sigma_2$, and finally perform the compression along $\Sigma_1$. Repeat until obtaining an incompressible Seifert surface $S_1$. Now any compressing disk for $\Sigma_2$ intersects $S_1$ in inessential curves that can similarly be removed by an innermost circle argument, so compress $\Sigma_2$ along disks disjoint from $S_1$ to obtain an incompressible surface $S_2$ whose interior is disjoint from $S_1$.

\end{proof}

\begin{proof}[Proof of Lemma \ref{lem:distance1}]
    If $d(S_1,S_2)=0$, then by Observation \ref{observation:coveringzero} the surfaces $S_1$ and $S_2$ are isotopic rel.\ boundary. Then up to isotopy rel.\ boundary in $S^3$, the surfaces $\Sigma_1,\Sigma_2$ are each obtained from $S_1$ by surgery along $n$ arcs for some $n$. By Lemma \ref{lem:stab}, $\Sigma_1$ and $\Sigma_2$ are isotopic rel.\ boundary in $B^4$.

    Now suppose $d(S_1,S_2)\le 1$. We will prove that $\Sigma_1,\Sigma_2$ are isotopic rel.\ boundary in $B^4$ by inducting on the genus $g$ of $\Sigma_1, \Sigma_2$ -- although for convenience, we will phrase this as a condition on Euler characteristic, inducting on Euler characteristic decreasing from the maximum possible value of $1$.

    If $\chi(\Sigma_i)=1$, then is a well-known consequence of the 3-dimensional Schoenflies theorem (see e.g.,\ discussion in \cite[Section 1.1]{hatcher_3m}) that the disks $\Sigma_1,\Sigma_2$ are isotopic rel.\ boundary in $S^3$ and hence in $B^4$.

Now suppose that Lemma \ref{lem:distance1} holds for Seifert surfaces of Euler characteristic strictly greater than $\chi(\Sigma_i)$. If $d(S_1,S_2)=0$ then $\Sigma_1$ and $\Sigma_2$ are isotopic rel.\ boundary in $B^4$, so assume that $d(S_1,S_2)=1$. By Corollary \ref{cor:distance1}, we may apply an isotopy rel.\ boundary in $S^3$ to $S_1,S_2$ so that they have disjoint interior. Then $S_1\cup S_2$ is an embedded surface $F$ that is essential with respect to $L$.

Choose a reduced, alternating diagram $\D$ for $L$. Lemma \ref{C.T.L} implies that $F$ includes a standard tube over some crossing $c$ of $\D$ as in Figure \ref{fig:proof} (left). Fixing $L$, isotope $S_1$ and $S_2$ so that their interiors agree precisely at a band $b$ above $c$ as in Figure \ref{fig:proof} (center). Let $\D'$ be the diagram obtained by smoothing $\D$ at $c$, respecting orientations, and let $L'$ be the link described by $\D'$. Since $\D$ is alternating, non-split, and reduced, the smoothed diagram $\D'$ is alternating and non-split.

Without loss of generality, assume $g(S_1)\ge g(S_2)$. Attach tubes to $S_2$ (away from $b$) to obtain a surface $\tilde{S}_2$ of the same genus as $S_1$. By Lemma \ref{lem:stab}, $\tilde{S}_2$ is isotopic rel.\ boundary in $B^4$ to $S_2\#(g(S_1)-g(S_2))T$, where $T$ is a standard torus in $S^4$.


Consider the surfaces $S'_1:=S_1\setminus b$, $S'_2:=\tilde{S}_2\setminus b$ bounded by the link $L'$. 
We have $\chi(S'_1)=\chi(S'_2)=\chi(S_1)+1>\chi(\Sigma_i)$. By Proposition \ref{prop:stilldisjoint}, $S'_1, S'_2$ compress to surfaces with disjoint interior and hence covering distance at most 1. Then by inductive hypothesis, we find that $S'_1$ and $S'_2$ are isotopic rel.\ boundary in $B^4$. Since $\Sigma_1=S'_1\cup b$ and $\tilde{S}_2=S'_2\cup b$, we find also that $S_1$ and $\tilde{S}_2$ are isotopic rel.\ boundary in $B^4$.

By Lemma \ref{lem:stab}, the surface 
$\Sigma_1$ is isotopic rel.\ boundary in $B^4$ to \[S_1\#(g(\Sigma_1)-g(S_1)) T.\] Similarly, $\Sigma_2$ is isotopic rel.\ boundary in $B^4$ to \begin{align*}S_2\# (g(\Sigma_2)-g(S_2))T&=(S_2\#(g(S_1)-g(S_2))T)\#(g(\Sigma_1)-g(S_1))T\\&=\tilde{S}_2\# (g(\Sigma_1)-g(S_1))T.\end{align*} Since $S_1$ and $\tilde{S}_2$ are isotopic rel.\ boundary in $B^4$, we conclude that $\Sigma_1$ and $\Sigma_2$ are isotopic rel.\ boundary in $B^4$ as claimed.
\end{proof}

\begin{remark}
    In the proof of Lemma \ref{lem:distance1}, it is essential that we are concluding isotopy in $B^4$ rel.\ boundary rather than in $S^3$. Even if $S'_1,S'_2$ were isotopic rel.\ boundary in $S^3$, an isotopy from $S'_1$ to $S_2'$ might pass $S'_1$ through the band $b$ and hence not extend to an isotopy from $S_1$ to $\tilde{S}_2$. Thus, the induction fails to prove that same-genus Seifert surfaces for non-split, alternating links are unique up to isotopy rel.\ boundary in $S^3$, even when we restrict to minimal genus Seifert surfaces. This is consistent with the fact that minimal genus Seifert surfaces for non-split, alternating links are generally {\emph{not}} unique up to isotopy in $S^3$, as discussed in Section \ref{sec:intro}. 
\end{remark}

We can now prove Theorem \ref{thm:main}.

\begin{mainthm}
     Let $\Sigma_1,\Sigma_2$ be same-genus Seifert surfaces for a non-split, alternating link. 
     Then $\Sigma_1$ and $\Sigma_2$ are isotopic rel.\ boundary in $B^4$.
\end{mainthm}

\begin{proof}


     Let $\chi:=\chi(\Sigma_i)$. Let $S_i$ be an incompressible surface obtained from $\Sigma_i$ by $k_i$ compressions. Assume $k_1\le k_2$, so $\chi\le \chi(S_1)\le \chi(S_2)$.

  By Theorem \ref{thm:filtration}, there exist incompressible Seifert surfaces $S_1=F_0,F_1,F_2,\ldots, F_n,F_{n+1}=S_2$ such that $\chi(F_j)\ge \chi(S_2)\ge\chi$ and the surfaces $F_j, F_{j+1}$ have disjoint interior for each $j\in\{0,\ldots, n\}$. For each $j=1,\ldots, n$, let $\Sigma'_j$ be a Seifert surface obtained from $F_j$ by attaching tubes to $F_j$ such that $\chi(\Sigma'_j)=\chi$ for all $j$. Let $\Sigma'_0:=\Sigma_1, \Sigma'_{n+1}:=\Sigma_2$.

  Now for each $j=0,\ldots, n$, the surfaces $\Sigma'_j, \Sigma'_{j+1}$ are same-genus Seifert surfaces for $L$ that compress to incompressible Seifert surfaces $F_j, F_{j+1}$. Note that $d(F_j, F_{j+1})\le 1$ since the interiors of $F_j, F_{j+1}$ are disjoint. Then By Lemma \ref{lem:distance1}, $\Sigma'_j, \Sigma'_{j+1}$ are isotopic rel.\ boundary in $B^4$. Since $\Sigma_1=\Sigma'_0, \Sigma_2=\Sigma'_{n+1}$, we conclude that $\Sigma_1,\Sigma_2$ are isotopic rel.\ boundary in $B^4$. 
\end{proof}

Corollaries \ref{cor1} and \ref{cor2} now follow easily.

\begin{cor1}
  Up to isotopy rel.\ boundary in the 4-ball, a non-split, alternating link $L$ bounds a unique minimal genus Seifert surface $S$. Any other Seifert surface for $L$ is isotopic rel.\ boundary in $B^4$ to a connected sum of $S$ and some number of standard tori in $S^4$.
\end{cor1}
\begin{proof}
Let $S$ be a minimal genus Seifert for $L$.
    By Theorem \ref{thm:main}, any other minimal genus Seifert surface for $L$ is isotopic to $S$ rel.\ boundary in $B^4$. Let $\Sigma$ be a Seifert surface for $L$, and let $S'$ be a genus-$g(\Sigma)$ Seifert surface obtained by attaching $n$ tubes to $S$. By Theorem \ref{thm:main}, $\Sigma$ is isotopic rel.\ boundary in $B^4$ to $S'$, which by Lemma \ref{lem:stab} is isotopic rel.\ boundary in $B^4$ to $S\#_n T$.
\end{proof}

\begin{cor2}
  Let $S_1, S_2$ be Seifert surfaces for a non-split, alternating link $L$ in $S^3$. Viewing $S^3$ as an equator of $S^4$, push the interiors of $S_1$ and $S_2$ off $S^3$ toward opposite normal directions of $S^3$ to obtain surface $S_1^+, S_2^-$ in $S^4$ with the same boundary but disjoint interior. The closed surface $S_1^+\cup S_2^-$ bounds a handlebody smoothly embedded in $S^4$, i.e.,\ is smoothly unknotted.
\end{cor2}

\begin{proof}
   Let $S$ be a minimal genus Seifert surface for $L$. Obtain a closed surface $F$ in $S^4$ by gluing two copies of $S$, with the interiors of each copy pushed off $S^3$ toward opposite sides. Then $F$ is isotopic to the boundary of the handlebody $S\times I$.

   By Theorem \ref{thm:main}, $S_1^+\cup S_2^-$ is isotopic to $F\#_n T$, where $n=(g(S_1)-g(S))+(g(S_2)-g(S))$. We conclude $S_1^+\cup S_2^-$ also bounds a handlebody.
\end{proof}

\section{Future open questions}

We end with some natural open questions for other families of links. We first suggest one consider almost-alternating links, i.e.,\ links that admit a diagram which becomes alternating after applying one crossing change. (Note the distinction between being almost-alternating vs.\ admitting a crossing change to some alternating link.) 


\begin{questions}\label{questions:almostalternating}\leavevmode
    \begin{enumerate}[label=(\alph*)]
        \item\label{q:item1} If $K$ is almost-alternating, must any two same-genus Seifert surfaces for $K$ be smoothly isotopic in $B^4$?
        \item\label{q:item2} If $K$ is almost-alternating, must the union of any two Seifert surfaces for $K$ yielding a closed surface in $S^4$ (as in Question \ref{q:seifertunknotting}) be smoothly unknotted?
        \end{enumerate}
        An affirmative answer to \ref{q:item1} implies an affirmative answer to \ref{q:item2}.
\end{questions}

Arguably, it would be more interesting if the answer to Question \ref{questions:almostalternating}\ref{q:item1} is ``no," as this would provide a sharpness result for Theorem \ref{thm:main} is sharp.

\begin{problem}
Find same-genus Seifert surfaces for an almost-alternating knot $K$ that are not isotopic in $B^4$. Even better: arrange for these surfaces to be minimal genus.
\end{problem}

Theorem \ref{thm:main} is the first result theorem proving uniqueness of minimal genus Seifert surfaces for a family of knots up to isotopy rel.\ boundary in $B^4$ without having uniqueness up to isotopy in $S^3$.

\begin{question}
    How many minimal genus Seifert surfaces does a given hyperbolic knot bound up to isotopy rel.\ boundary in $B^4$?
\end{question}

Hyperbolic knots may bound more than one minimal genus Seifert surface up to isotopy rel.\ boundary in $B^4$ \cite{hayden-kim-miller-park-sundberg}. However, a hyperbolic knot bounds a finite number of minimal genus Seifert surfaces in $S^3$ (and in fact, a finite number of Seifert surfaces of any fixed genus) according to constraints from normal surface theory \cite{schubert} (see also \cite[Theorem 4]{jpw}). Perhaps a similar analysis could be used to give a nontrivial explicit upper bound on the number of Seifert surfaces for a hyperbolic knot up to isotopy rel.\ boundary in $B^4$.

\bibliography{bib}
\bibliographystyle{alpha}

\end{document}